\documentclass[11pt,oneside,a4paper]{amsart}
\usepackage{graphicx}
\usepackage{amsthm}
\usepackage{hyperref}
\usepackage{amsmath}
\usepackage{amssymb}
\usepackage[T1]{fontenc}
\usepackage{color}
\usepackage{soul}

\usepackage{mathrsfs}

\newtheorem{lem}{Lemma}[section]
\newtheorem{thm}[lem]{Theorem}
\newtheorem{rem}[lem]{Remark}
\newtheorem{defin}[lem]{Definition}
\newtheorem{cor}[lem]{Corollary}
\newtheorem{example}[lem]{Example}
\newtheorem{prop}[lem]{Proposition}

\def \NN{\mathbb{N}}

\def \RR{\mathbb{R}}
\def \Rd{{\RR^d}}

\newcommand{\ind}{{\bf 1 }}

\definecolor{zm}{RGB}{128,128,0}

\newcommand{\Kz}{K}

\title[Sharp Gaussian estimates]{Sharp Gaussian estimates for heat kernels of~Schr{\"o}dinger operators}


\author{Krzysztof Bogdan}
 \address{Wroc{\l}aw University of Science and Technology,
Wybrze{\.z}e Wyspia{\'n}skiego 27,
50-370 Wroc{\l}aw, Poland}
\email{bogdan@pwr.edu.pl}

\author{Jacek Dziuba{\'n}ski}
\address{University of Wroc{\l}aw,
Pl. Grunwaldzki 2/4,
50-384 Wroc{\l}aw,
Poland}
\email{Jacek.Dziubanski@math.uni.wroc.pl}

\author{Karol Szczypkowski}
\address{Universit{\"a}t Bielefeld, Postfach 10 01 31,
D-33501 Bielefeld, Germany  and
Wroc{\l}aw University of Science and Technology,
Wybrze{\.z}e Wyspia{\'n}skiego 27,
50-370 Wroc{\l}aw, Poland
}
\email{karol.szczypkowski@math.uni-bielefeld.de, karol.szczypkowski@pwr.edu.pl}


\begin{document}

\begin{abstract}
We characterize 
functions $V\le 0$ for which the 
heat kernel of the Schr\"o\-dinger operator $\Delta+V$ is comparable with the  Gauss-Weierstrass kernel uniformly in space and time. 
In dimension $4$ and higher the condition 
turns out to be more restrictive
than the condition of the boundedness of the Newtonian potential of $V$. This resolves the question of V.~Liskevich and Y.~Semenov posed in 1998.
We also 
give specialized sufficient conditions for the comparability,
showing that 
local $L^p$ integrability of 
$V$ for $p>1$ is not necessary for the comparability.

\end{abstract}

\maketitle

\section{Introduction and main results}\label{sec:prel}

Let $d=1,2,\ldots$.
We consider the Gauss-Weierstrass kernel,
\[g(t,x,y)=
(4\pi t)^{-d/2} e^{-|y-x|^2/(4t)}, \qquad t>0,\ x,y\in\Rd.\]
It is well known that $g$ is the fundamental solution of the equation $\partial_t u=\Delta u$, and time-homogeneous probability transition density--the heat kernel of $\Delta$.
For
Borel measurable function $V: \Rd\to \overline\RR
$
we call 
$G:(0,\infty)\times \Rd\times \Rd\to [0,\infty]$
the heat kernel of $\Delta+V$
or the Schr\"odinger perturbation of $g$ by $V$, if the following
Duhamel or perturbation formula holds for $t>0$, $x,y\in \Rd$,
\[
G(t,x,y)=g(t,x,y)+\int_0^t \int_\Rd G(s,x,z)V(z)g(t-s,z,y)dzds.
\]
Under appropriate assumptions on $V$,
explicit definition of $G$ 
may be given by means of the Feynmann-Kac formula \cite[Section~6]{MR2457489}, the Trotter formula \cite[p.~467]{MR1978999}, the perturbation series \cite{MR2457489}, or by means of quadratic forms on $L^2$ spaces \cite[Section~4]{MR591851}. 
In particular the assumption $V\in L^p(\Rd)$ with $p>d/2$ was
used
by Aronson
\cite{MR0435594}, Zhang \cite[Remark~1.1(b)]{MR1978999} and by Dziuba{\'n}ski and Zienkiewicz \cite{MR2164260}.
Aizenman and Simon \cite{MR644024,MR670130} proposed
functions $V(z)$ from the Kato class, which
contains
$L^{p}(\Rd)$
for every $p>d/2$
\cite[Chapter~4]{MR644024}, see also Chung and Zhao \cite[Chapter 3, Example 2]{MR1329992}.
An enlarged Kato class was used by Voigt \cite{MR845197} in the study of Schr{\"o}dinger semigroups on  $L^1$ \cite[Proposition~5.1]{MR845197}.
For perturbations by time-dependent functions $V(u,z)$, Zhang \cite{MR1360232,MR1488344} introduced the so-called parabolic Kato condition.
The 
condition was then
generalized and employed by Schnaubelt and Voigt \cite{MR1687500}, Liskevich and Semenov
\cite{MR1783642}, 
Milman and Semenov \cite{MR1994762},
Liskevich, Vogt and Voigt
\cite{MR2253015}, and Gulisashvili and van Casteren \cite{MR2253111}.

Given 
the function $V: \Rd\to \RR$ we ask if there are positive numbers, i.e., {\it constants} $0<c_1\le c_2<\infty$ such that the following two-sided bound holds,
\begin{align}\label{est:sharp_uni}
c_1  \leq \frac{G(t,x,y)}{g(t,x,y)}\leq c_2,\qquad t>0,\ x,y\in \Rd.
\end{align}
One can also ponder a weaker property--if for a given $T\in (0,\infty)$,
\begin{align}\label{est:sharp_time}
c_1  \leq \frac{G(t,x,y)}{g(t,x,y)}\leq c_2 \,,\qquad 0<t\le T,\ x,y\in \Rd.
\end{align}
We call \eqref{est:sharp_uni} and \eqref{est:sharp_time} {\it sharp Gaussian estimates} or bounds, respectively {\it global} (or uniform) and {\it local} in time.
We observe that the inequalities in \eqref{est:sharp_uni} and \eqref{est:sharp_time} are stronger than
the plain
{\it 
Gaussian estimates:}
\begin{align*}
c_1\, (4\pi t)^{-d/2} e^{-\frac{|y-x|^2}{4t\varepsilon_1}} \leq G(t,x,y)\leq c_2\, (4\pi t)^{-d/2} e^{-\frac{|y-x|^2}{4t\varepsilon_2}},\quad x\in \Rd,
\end{align*}
where
$0<\varepsilon_1,c_1\le 1\le \varepsilon_2,c_2<\infty$, which
can also be global or local in time.

Berenstein proved
the plain Gaussian estimates for  $V\in L^p$ with $p>d/2$ (see \cite{MR1642818}).
Simon \cite[Theorem B.7.1]{MR670130} resolved
them for $V$ in the Kato class, Zhang \cite{MR1488344}  and Milman and Semenov \cite{MR1994762} applied
the parabolic Kato class for this purpose.
For further discussion we refer the reader to
\cite{MR1783642}, \cite{MR2253015}, \cite{MR1994762}, \cite{MR1978999} and \cite[Lemma~4]{MR3000465}. We also refer to  Bogdan and Szczypkowski \cite[Section~1, 4]{MR3200161} for a survey of the plain Gaussian bounds for Schr\"odinger heat kernels along with a streamlined approach, new results and explicit constants based on the so-called 4G inequality.

The plain Gaussian estimates are ubiquitous in analysis but \eqref{est:sharp_uni} and  \eqref{est:sharp_time} provide precious qualitative information, if they hold for $V$.
It is intrinsically difficult to characterize \eqref{est:sharp_uni} and  \eqref{est:sharp_time}
for those $V$ that take on positive values, while the case of $V\le0$ is more manageable.
Arsen'ev proved \eqref{est:sharp_time} for $V\in L^p+L^{\infty}$
with $p>d/2$, $d\geq 3$. Van Casteren \cite{MR1009389} proved
\eqref{est:sharp_time} for $V$ in the intersection of the Kato class and
$L^{d/2}+L^{\infty}$ for $d\geq 3$ (see \cite{MR1994762}). Arsen'ev also obtained \eqref{est:sharp_uni} for $V\in L^p$ with $p > d/2$ under additional smoothness assumptions (see \cite{MR1642818}). Liskevich and Semenov  
stated sufficient conditions for \eqref{est:sharp_uni} and \eqref{est:sharp_time} in \cite[Theorem 1, Corollary 1, Theorem 2]{MR1642818}.
Zhang \cite[Theorem~1.1]{MR1978999} and Milman and Semenov \cite[Theorem~1C, Remark (2)]{MR1994762}
gave  sufficient 
integro-supremal conditions for \eqref{est:sharp_uni}  and  \eqref{est:sharp_time}
for general $V$  and characterized 
\eqref{est:sharp_uni} and \eqref{est:sharp_time} 
for $V\le 0$. It will be convenient to state the conditions by means of
\begin{align*}
S(V,t,x,y)=\int_0^t \int_{\Rd} \frac{g(s,x,z)g(t-s,z,y)}{g(t,x,y)}|V(z)|\,dzds, \quad t>0,\ x,y\in \Rd.
\end{align*}
The motivation for using this quantity comes from
Zhang \cite[Lemma~3.1 and Lemma~3.2]{MR1978999} and from Bogdan, Jakubowski and Hansen  \cite[(1)]{MR2457489}.
We often write $S(V)$ if we do not need to specify $t,x,y$.  As explained in Section~\ref{sec:gen_app}, $S(V)$ is the potential of $|V|$ for 
the so-called Gaussian bridges.
We also note that \cite[Section~6]{MR2457489} uses $S(V)$ for general transition densities. 
The next two results indicate why $S(V)$ is important. Their proofs are given in Section~\ref{sprel}.
\begin{lem}\label{cor:gen_neg}
Let $V\leq 0$. Then \eqref{est:sharp_uni} is equivalent to
\begin{equation}\label{e:sSbi}
\sup_{t>0,\,x,y\in\Rd} S(V,t,x,y)<\infty.
\end{equation}
Also, for each $T\in (0,\infty)$,
\eqref{est:sharp_time} is equivalent to
\begin{equation}\label{e:sSbt}
\sup_{0<t\leq T,\,x,y\in\Rd} S(V,t,x,y)<\infty.
\end{equation}
\end{lem}
We say that $V$ satisfying \eqref{e:sSbi} or \eqref{e:sSbt} has bounded potential for bridges (is bridge-potential bounded)
globally or locally in time, respectively.

\begin{lem}\label{lem:gen_neg}
If for some
$h>0$ and $0\le \eta<1$ we have
$$\sup_{0<t\leq h,\,x,y\in\Rd} S(V^+,t,x,y)\le \eta,$$
and if $S(V^-)$ is bounded on bounded subsets of
$(0,\infty)\times\Rd\times\Rd$,
then
\begin{align}\label{gen_est}
e^{-S(V^-,t,x,y)} \leq \frac{G(t,x,y)}{g(t,x,y)}\leq \left(\frac{1}{1-\eta}\right)^{1+t/h}, \qquad t>0, \ x,y\in \Rd \,.
\end{align}
\end{lem}
We also notice the following consequence of 
 the Duhamel formula.
\begin{rem}\label{rem:Vp}
If $V \geq 0$, then \eqref{est:sharp_uni} implies 
\eqref{e:sSbi} and \eqref{est:sharp_time} implies
\eqref{e:sSbt}.  
\end{rem}
For clarity we note
that $S(V)$ is unbounded 
for all nontrivial $V$ in dimensions
$d = 1$ and $2$, hence \eqref{est:sharp_uni} is impossible for nontrivial $V\ge 0$ and nontrivial $V\le 0$ in these dimensions. 
This is explained at the end of Section~\ref{sprel} below.

\subsection{Characterization of sharp Gaussian estimates}

The 
conditions in Lemma~\ref{cor:gen_neg} and 
\ref{lem:gen_neg} may be cumbersome to verify.
For this reason we propose a simpler integro-supremal test
for \eqref{est:sharp_uni}.
For $d\geq 3$ and $x,y\in \Rd$ we define
\begin{eqnarray*}
\Kz  (V,x,y)&=&\int_{\Rd} |V(z)|\Kz (z-x,y)\,dz\,,
\end{eqnarray*}
where
\begin{align}\label{Kernel}
\Kz (x,y)= \frac{e^{-
\left(|x||y|-x\cdot y \right)/2}}{|x|^{d-2}} \left(1+|x||y| \right)^{d/2-3/2}\,,
\end{align}
and $x\cdot y$ is the usual scalar product. We denote, as usual,
\begin{align*}\label{def:sSbi}
\|S(V) \|_{\infty}&=\sup\limits_{t>0,\,x,y\in\Rd} S(V,t,x,y),\\
\|\Kz(V) \|_{\infty}&= \sup\limits _{x,y\in\Rd}
\Kz (V,x,y)\,.
\end{align*}
These two integro-supremal quantities
turn out to be comparable, as follows.
\begin{thm}\label{thm:J_0}
There are constants 
$M_1$, $M_2$  depending only on $d$, such that
\begin{equation}\label{eq:cSK0}
M_1 \|K(V)\|_{\infty} \leq 
\|S(V)\|_{\infty}
\leq M_2 \|K(V)\|_{\infty}\,.
\end{equation}
\end{thm}
The proof of Theorem~\ref{thm:J_0} is given in Section~\ref{sec:proofs}.
By \eqref{eq:cSK0} and Lemma~
\ref{cor:gen_neg} we get the following characterization of the sharp global Gaussian estimates.

\begin{cor}\label{cor:sGbbyK}
If $V\leq 0$, then \eqref{est:sharp_uni} holds if and only if $\Kz (V)$ is bounded.
\end{cor}

Similarly,  for general (signed) $V$ we get \eqref{est:sharp_uni} 
provided $\|\Kz (V^-)\|_\infty<\infty$ and $\|\Kz (V^+)\|_\infty<1/M_2$. This follows from  Lemma~\ref{lem:gen_neg} and Theorem~\ref{thm:J_0}.

We next elaborate on more specific applications of $K(V)$ to sharp global Gaussian estimates. In particular we resolve a 
long-standing open problem posed by Liskevich and Semenov.
For $d\geq 3$ we let $C_d=\Gamma(d/2-1)/(4\pi^{d/2})$. The
Newtonian kernel is $\int_0^{\infty}g(s,x,z)ds=C_d|z-x|^{2-d}$, $x,z\in \Rd$, and the Newtonian 
potential of (a nonnegative) function $f$ 
at $x\in \Rd$ is denoted
\begin{align*}
-\Delta^{-1} f(x)
:=\int_0^{\infty}\int_{\Rd}g(s,x,z) f(z)\,dzds
=\int_{\Rd} \frac{C_d}{|z-x|^{d-2}}f(z)\,dz\,.
\end{align*}
For $d=3$ the formula \eqref{Kernel}
considerably simplifies and we easily obtain
\begin{align}\label{eq:d_3}
\|\Kz (V)\|_\infty= C_{d}^{-1}\, \|\Delta^{-1} |V| \|_{\infty} \quad \mbox{ for } d=3\,.
\end{align}
Thus if $d=3$ and $V\leq 0$, then by Corollary~\ref{cor:sGbbyK}
the sharp global Gaussian bounds \eqref{est:sharp_uni} are equivalent to the condition of potential-boundedness, namely
$\|\Delta^{-1} V \|_{\infty}<\infty$. This  is classical
\cite[Remark~(3) on p. 4]{MR1994762} but remarkable, 
because the Newtonian kernel is isotropic, that is rotation-invariant, while $K(V)$ and $S(V)$ have a certain anisotropy-sensitivity.

Putting aside the exceptional, the main focus of the present paper is on 
$d\ge 4$. As usual, we let 
$$\|V \|_{d/2}=\left(\int_\Rd |V(z)|^{d/2}dz\right)^{2/d}.$$
By \cite[Theorem~2]{MR1642818} and \cite[Remark~(1) and (4) on p. 4]{MR1994762} we have \eqref{est:sharp_uni} for $d\geq 4$ if
$\|\Delta^{-1} V^- \|_{\infty}+\|V^-\|_{d/2}<\infty$, $\|\Delta^{-1} V^+ \|_{\infty}<1$ and
$\|V^+\|_{d/2}$  is
small enough.
A long-standing open problem for \eqref{est:sharp_uni} with $V\le 0$ posed in 1998 by Liskevich and Semenov \cite[p. 602]{MR1642818} reads as follows: ``The validity of the two-sided estimates for the case $d>3$ without the additional assumption $V\in L^{d/2}$ is an open question.'' 
In view of Theorem~\ref{thm:J_0} and Lemma~\ref{cor:gen_neg} the question is whether  for $d\ge 4$ the finiteness of $\|\Delta^{-1}V\|_{\infty}$ implies the finiteness of  
$\|K(V)\|_{\infty}$. Our next estimate
is a variant of \cite[Corollary~1]{MR1642818} and closely relates  $\|K(V)\|_{\infty}$ to $\|\Delta^{-1} V \|_{\infty}$ and $\|V\|_{d/2}$ for $d\geq 4$:
\begin{equation}\label{eq:KV}
C_d^{-1}\|\Delta^{-1}|V|\|_{\infty}\le
\|\Kz (V)\|_\infty \leq 2^{(d-3)/2}\Big( C_d^{-1}\|\Delta^{-1}|V|  \|_{\infty} + \kappa_d \|V \|_{d/2}\Big).
\end{equation}
In Section~\ref{sec:proofs} we prove \eqref{eq:KV}  
and the following result, which points out  a gap between $\|\Kz (V)\|_\infty$ and $\|\Delta^{-1}|V|\|_{\infty}$ in  \eqref{eq:KV}. 
\begin{prop}\label{thm:dist}
Let $d\geq 4$. 
For $z=(z_1, z_2,\ldots,z_d)\in\Rd$ we write $z=(z_1,\mathbf z_2)$, 
where $\mathbf z_2=(z_2,\ldots,z_d)\in \RR^{d-1}$. 
We define
\begin{eqnarray*}
A&=&\{(z_1,\mathbf z_2)\in \Rd : \ z_1>4, \ |\mathbf z_2|\leq \sqrt{z_1}\}, \quad \mbox{ and }\\
V(z_1,\mathbf z_2)&=&-\frac{1}{z_1} \ind_{A}(z_1,\mathbf z_2).
\end{eqnarray*} 
Then $\| \Delta^{-1} V \|_{\infty}<\infty$ but
$\|\Kz (V)\|_\infty=\infty$.
There is even a function $V \le  0$ with compact support such that $\|\Delta^{-1}V\|_{\infty}<\infty$ but $\|\Kz (V)\|_\infty=\infty$.
\end{prop}
From Lemma~\ref{cor:gen_neg} we conclude that
for $d\ge 4$ neither finiteness nor smallness of $\|\Delta^{-1} V \|_{\infty}$
are sufficient for \eqref{est:sharp_uni}. Therefore
the answer to the question of Liskevich and Semenov is negative.

Here are a few more comments that relate our result to existing literature and serve as preparation for the proofs.
Due to the work of Zhang \cite{MR1978999}, the following quantity
is a proxy
for $S(V)$,
\begin{align}
N(V,t,x,y)&:=\int_0^{t/2}\int_{\Rd}\frac{e^{-|z-y+(\tau/t)(y-x)|^2/(4\tau)}}{\tau^{d/2}}|V(z)|dzd\tau \nonumber\\
&+\int_{t/2}^t\int_{\Rd} \frac{e^{-|z-y+(\tau/t)(y-x)|^2/(4(t-\tau))}}{(t-\tau)^{d/2}} |V(z)|dzd\tau= N(V,t,y,x) \,,\label{def:N}
\end{align}
where $t>0$, $x,y\in \Rd$. Indeed,
by \cite[Lemma 3.1, Lemma 3.2, and line 11 on p. 469]{MR1978999}
there  are constants $m_1,m_2$ depending only on $d$ such that
\begin{align}
S(V,t,x,y)&\geq m_1\, N(V,t/2,x,y)\,,\qquad t>0, \ x,y\in \Rd \,,\tag{L} \label{L}\\
S(V,t,x,y)&\leq m_2\, N(V,t,x,y)\,,\qquad t>0, \ x,y\in \Rd \,.\tag{U} \label{U}
\end{align}
We also let $\|N(V) \|_{\infty}=\sup\limits_{t>0,\,x,y\in\Rd} N(V,t,x,y)$. By \eqref{L} and \eqref{U} we get
\begin{displaymath}
m_1\|N(V) \|_{\infty}\le
\|S(V) \|_{\infty}\le
m_2\|N(V) \|_{\infty}.
\end{displaymath}
In \cite[Theorem~1C]{{MR1994762}} and \cite[(8)]{MR1642818}
another quantity was used to study \eqref{est:sharp_uni} and \eqref{est:sharp_time},
\begin{align*}
e_*(V,\lambda)
&=\sup_{\alpha\in\Rd}\|V (\lambda-\Delta+2\alpha\cdot\nabla)^{-1}\|_{1\to 1}\\
&=\sup_{\alpha \in \Rd}\|(\lambda-\Delta+2\alpha\cdot\nabla )^{-1}|V|\|_{\infty}\,,
\end{align*}
where $\lambda\ge 0$. It may be given in terms of the Gauss-Weierstrass kernel, e.g.,
\begin{align*}
e_*(V,0)
&=(4\pi)^{-d/2}\sup_{x,y\in\Rd} \int_{\Rd} J(z-x,y) |V(z)|\,dz\,, \label{def:egJ}
\end{align*}
where
$$J(x,y)=\int_0^{\infty} \tau^{-d/2} e^{-\frac{|x-\tau y|^2}{4\tau}} d\tau\,, \qquad x,y\in\Rd.$$ 
There is a certain anisotropy-sensitivity of $e_*(V,0)$ due to $2\alpha\cdot\nabla$ above, which  is similar to 
that of $K(V)$. In fact,
in Lemma~\ref{lem:J_0} below we prove that
there are constants $c_1$, $c_2$ depending only on $d\ge 3$ such that
\begin{equation}\label{eq:ces}
c_1 \|\Kz (V)\|_\infty \leq e_*(V,0)\leq c_2 \|\Kz (V)\|_\infty\,.
\end{equation}
In view of Theorem~\ref{thm:J_0}, the quantities $\|S(V)\|_{\infty}$, $\|K(V)\|_{\infty}$, $\|N(V)\|_{\infty}$, $e_*(V,0)$ are all comparable, which makes them equivalent for studying \eqref{est:sharp_uni} with $V\leq 0$.

We 
add a few comments on the exceptional case $d=3$. By \cite[(3) in Remark on p. 4]{MR1994762}
and \eqref{eq:d_3} we  have 
$
e_*(V,0)=\|\Delta^{-1} |V| \|_{\infty}= C_{d}\|\Kz (V)\|_\infty$.
Also, for $d=3$ 
by  \cite[
Remark~(1) and (3) on p.~4]{MR1994762}, the condition
\begin{equation}\label{eq:mn1}
\|\Delta^{-1} V^- \|_{\infty}<\infty,\quad
\|\Delta^{-1} V^+ \|_{\infty}<1,
\end{equation}
suffices for \eqref{est:sharp_uni}.
Furthermore, if 
$V\le 0$, then 
the condition $
\|\Delta^{-1}V\|_{\infty}<\infty$ characterizes 
the plain global Gaussian bounds, see \cite{MR1456565} and \cite[p. 556 and Corollary~A]{MR1772429}.
Therefore by \eqref{eq:d_3}, for $d=3$ 
the plain global Gaussian bounds hold for $V\le 0$ if and only if the sharp global Gaussian bounds hold.
In contrast, for $d\geq 4$ by Proposition~\ref{thm:dist} the plain global Gaussian bounds may occur in the absence of the sharp global Gaussian bounds \eqref{est:sharp_uni}.

 We recall that $\sup_{0<t\leq T,\,x,y\in\Rd} S(V,t,x,y)$ with $T<\infty$ is useful for the local in time sharp Gaussian estimates \eqref{est:sharp_time}, see
Lemma~\ref{cor:gen_neg} and \ref{lem:gen_neg}.
In a similar fashion  $\sup_{0<t\leq T,\,x,y\in\Rd} N(V,t,x,y)$
is used in \cite[Theorem~1.1]{MR1978999},
while in \cite[Theorem~1C]{MR1994762} the authors
make use of $e_*(V,\lambda)$ for $\lambda>0$.
In this connection see also 
Corollary~\ref{prop:lower_exp} below.

\subsection{Sufficient conditions for sharp Gaussian estimates}

In this section we propose sufficient conditions for \eqref{est:sharp_uni} and \eqref{est:sharp_time} for functions $V$ which have a form of the tensor product.
Such conditions are the second main topic of the paper--they culminate in Theorem~\ref{prop:most_extended} below. We also 
show that
$L^p$ integrability for $p>1$ is not necessary for \eqref{est:sharp_uni} or \eqref{est:sharp_time}.
Let
$p, p_1, p_2\in [1,\infty]$.
\begin{defin}\label{d:tp}
We write
$f\in L^{p_1}(\RR^{d_1})\times L^{p_2}(\RR^{d_2})$ if there are $f_1\in L^{p_1}(\RR^{d_1})$ and $f_2\in L^{p_2}(\RR^{d_2})$, such that
$$
f(x_1,x_2)=f_1(x_1) f_2(x_2)\,,\qquad x_1\in\RR^{d_1}, \ x_2\in\RR^{d_2}.
$$
\end{defin}
\noindent
We note that $L^{p}(\RR^{d_1})\times L^{p}(\RR^{d_2}) \subset
L^{p}(\RR^{d_1+d_2})$, in fact
$\|f\|_p=\|f_1\|_p\|f_2\|_p$ if  $f$ is the tensor product $f (x_1,x_2)=f_1(x_1)f_2(x_2)$.
\vspace{5pt}
\begin{thm}\label{prop:most_extended}
Let $d_1, d_2\in \NN$, $d=d_1+d_2$, $V\colon \Rd\to \RR$,
$p_1, p_2\in [1,\infty]$
and
$$
\frac{d_1}{2p_1}+\frac{d_2}{2p_2}=1\,.
$$
\begin{enumerate}
\item[\rm (a)] If $r\in(p_1,\infty]$ and $V\in L^{r}(\RR^{d_1})\times L^{p_2}(\RR^{d_2})$, then
$$
\sup_{x,y\in\Rd} S(V,t,x,y)\leq c\, t^{1-d_1/(2r)-d_2/(2p_2)}\,,
$$
where $c=C(d_1,r)C(d_2,p_2) \frac{[\Gamma(1-d_1/(2r)-d_2/(2p_2))]^2}{\Gamma(2-d_1/r-d_2/p_2)} \|V_1 \|_{r}\|V_2 \|_{p_2}$.
\item[\rm(b)]
If $1\le q <p_1<r\le \infty$ and
$V\in \left[L^{q}(\RR^{d_1})\cap L^{r}(\RR^{d_1}) \right]\!\!\times L^{p_2}(\RR^{d_2})$, then
\eqref{e:sSbi} holds.
\end{enumerate}
\end{thm}
The proof of Theorem~\ref{prop:most_extended}
is given in Section~\ref{sec:gen_app}, where we use in a crucial way the tensorization of the Gauss-Weierstrass kernel and its bridges. 
Lemma~\ref{cor:gen_neg} and \ref{lem:gen_neg} provide the following conclusion.
\begin{cor}\label{cor:escb}
Under the assumptions of
{\rm Theorem~\ref{prop:most_extended}(a)},
$G$ 
satisfies the sharp local Gaussian bounds \eqref{est:sharp_time}. If $V\le 0$ and the assumptions of
{\rm Theorem~\ref{prop:most_extended}(b)}
hold, then $G$ has the sharp global Gaussian bounds \eqref{est:sharp_uni}.
\end{cor}
Clearly, if $|U|\le |V|$, then $S(U)\le S(V)$. This may be used to
extend
the 
conclusions of Theorem~\ref{prop:most_extended} and Corollary~\ref{cor:escb}
beyond tensor products.

\begin{prop}\label{cor:ce}
For every $d\ge 3$ there is a function $V\le 0$ such that \eqref{est:sharp_uni} holds but
$V\notin L^1(\Rd)\cup \bigcup_{p>1}L^p_{loc}(\Rd)$. 
\end{prop}
In particular \eqref{est:sharp_uni} does not necessitate $\|V \|_{d/2}<\infty$, i.e.,
the finiteness of $\|\Kz (V)\|_\infty$ does not imply that of  $\|V \|_{d/2}$; see also \eqref{eq:KV} in this connection.
We note in passing that local $L^1$ integrability is necessary for \eqref{est:sharp_time} if $V$ does not change sign, cf. Lemma~\ref{cor:gen_neg} and \ref{l:b}, and Remark~\ref{rem:Vp}.
The function $V$ in Proposition~\ref{cor:ce} is constructed in Section~\ref{sec:e} from highly anisotropic tensor products of power functions.

The structure of the remainder of the paper is as follows.
In Section~\ref{sprel} we provide definitions and preliminaries,
in particular we prove 
Lemma~\ref{cor:gen_neg} and \ref{lem:gen_neg}.
In Section~\ref{sec:proofs} we prove 
Theorem~\ref{thm:J_0}
and
Proposition~\ref{thm:dist}.
In Section~\ref{sec:gen_app} we 
prove Theorem~\ref{prop:most_extended}.
In Section~\ref{sec:e} we prove Proposition~\ref{cor:ce} 
and give examples which illustrate and comment on our results.

\section{Preliminaries}\label{sprel}

We let $\NN=\{1,2,\ldots\}$, $f^+=\max\{0,f\}$ and $f^-=\max\{0,-f\}$. 
Recall that $d\in \NN$ and $V$ is an arbitrary Borel measurable function$:\Rd\to \RR$.

We
begin with
the following observations on
integrability
and potential-boundedness \eqref{e:si} of  functions $V$ which are bridges potential-bounded.
\begin{lem}\label{l:b}
If $S(V,t,x,y)<\infty$ for some $t>0$, $x,y\in\Rd$, then $V\in L^1_{\rm loc}(\Rd)$.
If \eqref{e:sSbt} holds, then
\begin{align}\label{e:si_local}
\sup_{x\in\Rd}\int_0^T \int_{\Rd} g(s,x,z)|V(z)|\,dzds<\infty\,.
\end{align}
If  \eqref{e:sSbi} even holds,
then 
\begin{equation}\label{e:si}
\sup_{x\in\Rd}\int_0^{\infty} \int_{\Rd} g(s,x,z)|V(z)|\,dzds<\infty\,.
\end{equation}
\end{lem}
\begin{proof}
The first statement follows, because
$g(t,x,y)$ is locally bounded from below on $(0,\infty)\times\Rd\times\Rd$ (see
\cite[Lemma~3.7]{MR3713578} for a
quantitative general result).
Since
$\int_\Rd S(V,t,x,y)g(t,x,y)\,dy=\int_0^t \int_{\Rd} g(s,x,z)|V(z)|\,dzds$,
we see that
\eqref{e:sSbt} implies \eqref{e:si_local}
and
\eqref{e:sSbi} implies \eqref{e:si}.
\end{proof}

We shall 
use the following functions:
\begin{align*}
f(t)&=\sup_{x,y\in\Rd}S(V,t,x,y)\,,\qquad t\in(0,\infty),\\
F(t)&=\sup_{0<s<t}f(s)=\sup_{\substack{x,y\in\Rd\\ 0<s< t}}S(V,s,x,y)\,,\qquad t\in(0,\infty]\,.
\end{align*}
We fix $V$ and $x,y\in \Rd$. For $0<\varepsilon<t$, we consider
\begin{align*}
S(V,t-\varepsilon,x,y)=\int_0^t \int_{\Rd} \frac{g(s,x,z)g(t-\varepsilon-s,z,y)}{g(t-\varepsilon,x,y)}|V(z)|\,{\bf 1}_{[0,t-\varepsilon]}(u)\,dzds.
\end{align*}
By Fatou's lemma we get
$$
S(V,t,x,y)\le \liminf_{\varepsilon\to 0} S(V,t-\epsilon,x,y),
$$
meaning that $(0,\infty)\ni t\mapsto S(V,t,x,y)$ is lower semicontinuous on the left.
It follows that
$f$ is lower semi-continuous on the left, too.
In consequence, $f(t)\leq F(t)$ and $F(t)=\sup_{0<s\leq t} f(s)$ for $0<t<\infty$.

We next claim that $f$ is sub-additive, that is,
\begin{align}\label{ineq:S_chapm}
f(t_1+t_2)\leq f(t_1)+f(t_2)\,,\qquad t_1,\,t_2>0\,.
\end{align}
This follows from the Chapman-Kolmogorov equations for $g$.
Indeed, we have $S(V,t_1+t_2,x,y)=I_1+I_2$, where
\begin{align*}
I_1&=\int_0^{t_1} \int_{\Rd} \frac{g(s,x,z)g(t_1+t_2-s,z,y)}{g(t_1+t_2,x,y)}|V(z)|\,dzds\\
&=\int_0^{t_1} \int_\Rd\int_{\Rd} \frac{g(s,x,z)g(t_1-s,z,w)g(t_2,w,y)g(t_1,x,w)}{g(t_1+t_2,x,y)g(t_1,x,w)}|V(z)|\,dwdzds\\
&\le \int_\Rd \frac{g(t_2,w,y)g(t_1,x,w)}{g(t_1+t_2,x,y)} S(V,t_1,x,w)\,dw\le f(t_1)\,,
\end{align*}
and $I_2$ equals
\begin{align*}
&\int_{t_1}^{t_1+t_2} \int_{\Rd} \frac{g(s,x,z)g(t_1+t_2-s,z,y)}{g(t_1+t_2,x,y)}|V(z)|\,dzds
\\
&=\int_{t_1}^{t_1+t_2} \!\!\!\! \int_\Rd\int_{\Rd} \frac{g(t_1,x,w)g(s-t_1,w,z)g(t_2-(s-t_1),z,y)g(t_2,w,y)}{g(t_1+t_2,x,y)g(t_2,w,y)}|V(z)|\,dwdzds\\
&\le \int_\Rd \frac{g(t_1,x,w)g(t_2,w,y)}{g(t_1+t_2,x,y)} S(V,t_2,w,y)\,dw
\le f(t_2)\,.
\end{align*}
This yields \eqref{ineq:S_chapm}.
\begin{lem}\label{lemaF(t)}
For all $t,h>0$ we have
$f(t)
\leq
F(h)+ t\, f(h)/h.
$
\end{lem}
\begin{proof}
Let $k\in \NN$ be such that $(k-1)h<t\leq kh$, and let $\theta=t-(k-1)h$.
Then $t=\theta+(k-1)h$, and by
\eqref{ineq:S_chapm} we get
$$
f(t)\leq f(\theta)+ t\,f(h)/h \,\leq F(h)+t\, f(h)/h\,,
$$
since $0<\theta\leq h$.
\end{proof}
\begin{cor}\label{cor:ineq_most}
$F(t)\leq F(h)+ t\, F(h)/h$ and $F(2t)\le 2F(t)$ for $t,h>0$.
\end{cor}
We may now prove 
Lemma~\ref{lem:gen_neg} and Lemma~\ref{cor:gen_neg}.

\begin{proof}[Proof of Lemma~\ref{lem:gen_neg}]
If $V\le 0$ then $0\le G\le g$ is constructed in \cite[p. 470]{MR1978999}, and the Duhamel formula follows from  the finiteness of $S(V^-)$ and the discussion after \cite[(3.3)]{MR1978999}.
Then the left-hand side of \eqref{gen_est} follows from \cite[pp. 467-468]{MR1978999}, or we can use
\cite[(41)]{MR2457489}, which results therein from Jensen's inequality and the second displayed formula on page 252 of \cite{MR2457489}.
For general, i.e., signed $V$ the kernel $G$ is constructed by applying the above procedure to $g$ and $-V^-$, and then perturbing the resulting kernel by $V^+$. The latter is done by means of the perturbation series, cf. \cite[Lemma 2]{MR2457489};
then the Duhamel formula obtains without further conditions.
We now  prove the right hand side of \eqref{gen_est}, 
and without loss of generality we may assume that $V\ge 0$. 
 For $0<s<t$, $x,y\in\mathbb R^d$, we let $p_0(s,x,t,y)=g(t-s,x,y)$ and $p_n(s,x,t,y)=\int_s^t \int_{\mathbb R^d} p_{n-1}(s,x,u,z)V(z)p_0(u,z,t,y)\, dz\, du$, $n\in \NN$. Let $Q:\mathbb R\times \mathbb R\to [0,\infty)$ satisfy $Q(u,r)+Q(r,v)\leq Q(u,v)$. By \cite[Theorem 1]{MR2507445} (see also \cite[Theorem~3]{MR3000465}) if there is $0<\eta <1$ such that
\begin{equation}\label{condition1}
p_1(s,x,t,y)\leq [\eta + Q(s,t)]p_0(s,x,t,y),
\end{equation}
then
\begin{equation}\label{eq2}
\tilde p(s,x,t,y):=\sum_{n=0}^\infty p_n(s,x,t,y)\leq \Big(\frac{1}{1-\eta}\Big)^{1+\frac{Q(s,t)}{\eta}}p_0(s,x,t,y)\,. \end{equation}
Corollary \ref{cor:ineq_most} and the assumptions of the lemma imply that (\ref{condition1}) is satisfied with $
\eta=F(h)< 1$ and
$Q(s,t)=(t-s)F(h)/h
$. Since $G(t,x,y)=\tilde p(0,x,t,y)$, the proof of \eqref{gen_est} is complete (see also \cite[(17)]{MR2457489}).
\end{proof}
\begin{proof}[Proof of Lemma~\ref{cor:gen_neg}]
If \eqref{est:sharp_time} holds then Duhamel formula and nonnegativity of $G$ yield \eqref{e:sSbt}.
Similarly, \eqref{est:sharp_uni} implies \eqref{e:sSbi}. The reverse implications follow from \eqref{gen_est}.
\end{proof}

As a consequence of Corollary~\ref{cor:ineq_most} we also obtain the following result.
\begin{cor}\label{prop:lower_exp}
Let $V\leq 0$ and $T>0$.
Then \eqref{est:sharp_time} holds
if and only if
\begin{align}\label{ineq:lower_exp}
C e^{-ct} g(t,x,y)\leq G(t,x,y) \,,\qquad t>0,\,x,y\in\Rd\,,
\end{align}
for some constants $C$ and $c$.
In fact we can take
$$\ln C=-\sup_{\substack{x,y\in\Rd\\ 0<t\leq T}}S(V,t,x,y) \qquad {\rm and}\qquad c=\frac1T \sup_{x,y\in\Rd} S(V,T,x,y)\,.$$
\end{cor}
\begin{proof}
\eqref{ineq:lower_exp} implies \eqref{est:sharp_time} for every fixed $T>0$. Conversely, if \eqref{est:sharp_time} holds for fixed $T>0$,
then by Lemma~\ref{lem:gen_neg} and \ref{lemaF(t)} we have
$$\frac{G(t,x,y)}{g(t,x,y)}\geq
e^{-S(V,t,x,y)}\ge
e^{-f(t)}\geq e^{-F(T)} e^{-tf(T)\slash T}.$$
\end{proof}
\noindent
We note in passing that the above proof shows that \eqref{est:sharp_time} is determined by the behavior of $\sup_{x,y\in\Rd}S(V,t,x,y)$ for small $t>0$.
We also see that \eqref{e:si} and thus \eqref{e:sSbi}
fail in dimensions $d=1$ and $d=2$, because then
$\int_0^{\infty} g(s,x,z)ds\equiv \infty$, unless 
$V=0$ $a.e.$
From Lemma~\ref{cor:gen_neg} and Remark~\ref{rem:Vp} it follows that
\eqref{est:sharp_uni} fails for
nontrivial $V\le 0$ and for nontrivial $V\ge 0$ if $d=1$ or $2$.

\section{Characterization of the sharp global Gaussian estimates}\label{sec:proofs}

In this section we prove our main result, i.e., Theorem~\ref{thm:J_0}. We start by using $N(V,t)$, \eqref{U} and \eqref{L}, to estimate $S(V,t)$.
\begin{lem}\label{lem:upr}
Let $t>0$. We have
\begin{align*}
\int_0^{t/2}\int_{\Rd}\frac{e^{-|z-y+(\tau/t)(y-x)|^2/(4\tau)}}{\tau^{d/2}}|V(z)|\,dzd\tau  \leq N(V,t,x,y)\,,\quad x,y\in \Rd\,,
\end{align*}
and
\begin{align*}
\sup_{x,y}N(V,t,x,y) \leq 2 \,\sup_{x,y} \int_0^{t/2}\int_{\Rd}\frac{e^{-|z-y+(\tau/t)(y-x)|^2/(4\tau)}}{\tau^{d/2}}|V(z)|\,dzd\tau\,.
\end{align*}
\end{lem}
\begin{proof}
The first inequality follows by the definition of $N(V,t,x,y)$.
For the proof of the second one we note that
\begin{align*}
\int_{t/2}^t\int_{\Rd} &\frac{e^{-|z-y+(\tau/t)(y-x)|^2/(4(t-\tau))}}{(t-\tau)^{d/2}} |V(z)|dzd\tau\\
&\qquad = \int_0^{t/2}\int_{\Rd} \frac{e^{-|z-x+(\tau/t)(x-y)|^2/(4\tau)}}{\tau^{d/2}} |V(z)|dzd\tau\,.
\end{align*}
\end{proof}

\begin{lem}\label{lem:J_0}
We have
\begin{align*}
\sup_{t>0,\,x,y\in\Rd} S(V,t,x,y) \geq m_1(4\pi)^{d/2} e_*(V,0)
\,,
\end{align*}
and
\begin{align*}
\sup_{t>0,\,x,y\in\Rd} S(V,t,x,y)\leq 2\, m_2 (4\pi)^{d/2} 
e_*(V,0)
\,.
\end{align*}
\end{lem}

\begin{proof}
By \eqref{L} and Lemma~\ref{lem:upr},
\begin{align*}
\sup_{t>0,\,x,y\in\Rd}& S(V,t,x,y)\geq m_1 \sup_{t>0,\,x,y\in\Rd} N(|V|,t/2,x,y)\\
&\geq m_1 \sup_{t>0,\,x,y\in\Rd} \int_0^{t/4}\int_{\Rd}\frac{e^{-|z-y+(2\tau/t)(y-x)|^2/(4\tau)}}{\tau^{d/2}}|V(z)|\,dzd\tau\\
&= m_1\sup_{t>0,\,y,w\in\Rd} \int_0^{t/4}\int_{\Rd}\frac{e^{-|z-y+\tau w|^2/(4\tau)}}{\tau^{d/2}}|V(z)|\,dzd\tau\\
&= m_1 \sup_{y,w\in\Rd} \int_{\Rd} J(z-y,w) |V(z)|\,dz=m_1(4\pi)^{d/2}  e_*(V,0)\,.
\end{align*}
By \eqref{U} and Lemma~\ref{lem:upr},
\begin{align*}
\sup_{t>0,\,x,y\in\Rd}& S(V,t,x,y)\leq m_2\sup_{t>0,\, x,y\in\Rd} N(V,t,x,y)\\
&\leq 2\, m_2 \sup_{t>0,\, x,y\in\Rd} \int_0^{t/2}\int_{\Rd}\frac{e^{-|z-y+(\tau/t)(y-x)|^2/(4\tau)}}{\tau^{d/2}}|V(z)|\,dzd\tau\\
&\leq 2\, m_2 \sup_{t>0,\, y,w\in\Rd} \int_0^{t/2}\int_{\Rd}\frac{e^{-|z-y+\tau w|^2/(4\tau)}}{\tau^{d/2}}|V(z)|\,dzd\tau\\
&= 2 \, m_2 \sup_{y,w\in\Rd} \int_{\Rd} J (z-y,w) |V(z)|\,dz
=2 \, m_2(4\pi)^{d/2}  e_*(V,0)\,.
\end{align*}
\end{proof}

\begin{proof}[Proof of Theorem~\ref{thm:J_0}]

For $x,w\in \Rd$ and $\tau>0$ we have
$$
\frac{|x-\tau w|^2}{4\tau}= \frac{|x|^2}{4\tau} - \frac{x\cdot w}{2}+ \frac{\tau|w|^2 }{4} \,.
$$
We change the variables $\tau =  4u/|w|^2 $ and use \cite[8.432, formula 6.]{MR3307944} to get
\begin{align*}
J(x,w)=
\int_0^{\infty} \tau^{-d/2} e^{-\frac{|x-\tau w|^2}{4\tau}} d\tau
= 
2 e^{\frac{x\cdot w}{2}} \left(\frac{|x|}{|w|}\right)^{-d/2+1}  K_{d/2-1}\left(\frac{|x||w|}{2}\right)\,.
\end{align*}
Here, as usual, $K_{\nu}$ is the modified Bessel function of the second kind.
We claim that for each $\nu \geq 1/2$,
$$
K_{\nu}(z)\approx z^{-\nu} e^{-z} (1+z)^{\nu-1/2}\,, \qquad z>0\,.
$$
Here $\approx$ means that the ratio of both sides is bounded above and below by constants independent of $z$.
The comparison is obtained by putting $x=1$ in \cite[8.432, formulas 9. and 8.]{MR3307944} and considering cases $z\le1$ and $z> 1$, correspondingly.
Therefore,
\begin{align*}
J(x,w)&\approx  e^{-\left(|x||w|-x\cdot w \right)/2}
 \left(\frac{|x|}{|w|}\right)^{-d/2+1}
 \left(\frac{|x||w|}{2}\right)^{-d/2+1}\left(1+\frac{|x||w|}{2}\right)^{d/2-3/2}\\
 &\approx K(x,w)\,,
\end{align*}
and so $K(V)\approx e_*(V,0)$. The result follows by Lemma~\ref{lem:J_0}.
\end{proof}

\begin{proof}[Proof of \eqref{eq:KV}]
The left hand side inequality follows from the identity
$\Kz  (V, x,0)= C_d^{-1} (-\Delta^{-1})|V|(x)$.
If $y=0$, then the upper bound trivially holds. For $y\neq 0$
we consider two domains of integration. We have
\begin{align*}
\int_{|z-x||y|\leq 1} &\Kz (z-x,y)|V(z)|\,dz
\leq 2^{(d-3)/2} \int_{|z-x||y|\leq 1} \frac{1}{|z-x|^{d-2}} |V(z)|dz\\
&\leq \frac{2^{(d-3)/2}}{C_d} |\!|\Delta^{-1}|V| |\!|_{\infty}\,.
\end{align*}
Furthermore, by a change of variables and the H{\"o}lder inequality,
\begin{align*}
&\int_{|z-x||y|\geq 1} \Kz (z-x,y)|V(z)|dz\\
&\leq 2^{(d-3)/2} \!\!\! \int_{|z-x||y|\geq 1}
\frac{e^{-\frac1{2}\left(|z-x||y|-(z-x)\cdot y \right)}}{(|z-x||y|)^{(d-1)/2}} |y|^{d-2}|V(z)|dz
\leq 2^{(d-3)/2}\kappa_d  |\!|V |\!|_{d/2}
\,,
\end{align*}
where
$$
\kappa_d=\left(
\int_{|w|>1} 
\left(e^{-\frac1{2}(|w|-w_1 )}|w|^{-(d-1)/2}\right)^{d/(d-2)} dw \right)^{(d-2)/d}<\infty \,.
$$
We skip the proof of the finiteness of $\kappa_d$; it can be found in the first version of the paper on arXiv: 1706.06172v1.
\end{proof}

\begin{proof}[Proof of Proposition~\ref{thm:dist}]
We first prove that $\|\Kz (V)\|_\infty=\infty$.
Let $y=(1,\mathbf 0)\in \mathbb R^d$.
For $z=(z_1, \mathbf z_2)\in A$ we have
$$ 0\leq |z||y|-z\cdot y=|z|-z_1 = \frac{|\mathbf z_2|^2}{\sqrt{z_1^2+|\mathbf z_2|^2}+z_1}\leq \frac{z_1}{\sqrt{z_1^2+|\mathbf z_2|^2}+z_1}\leq 1 $$
and thus also
$ z_1\le |z| \le 2z_1$. Then,
\begin{align*}
\|\Kz (V)\|_\infty&\geq \int_{\mathbb R^d} e^{-\frac{1}{2}(|z||y|-z\cdot y)}|V(z)|\frac{(1+|z||y|)^{\frac{d}{2}-\frac{3}{2}}}{|z|^{d-2}}\, d z
\geq c\int_A \frac{1}{z_1}\frac{z_1^{\frac{d}{2}-\frac{3}{2}}}{z_1^{d-2}}\, d z\\
&=c\int_{4}^\infty \int_{|\mathbf z_2|<\sqrt{z_1}} z_1^{-1+2-d+\frac{d}{2}-\frac{3}{2}}\, d\mathbf z_2\, dz_1\\
&=c\int_{4}^\infty  z_1^{-1+2-d+\frac{d}{2}-\frac{3}{2}+\frac{1}{2}(d-1)}\, dz_1
=\infty. 
\end{align*}
We now prove that $\|\Delta^{-1} |V| \|_{\infty}<\infty$.
By the symmetric rearrangement inequality (see \cite[Chapter~3]{MR1817225}) we have
\begin{align*}
\sup_{x\in \mathbb R^d}\int_{\mathbb R^d} \frac{1}{|z-x|^{d-2}} |V(z)|\, d z
= \sup_{x_1\in \mathbb R} \int_4^{\infty} \int_{\RR^{d-1}} \frac{\ind_{|\mathbf z_2|<\sqrt{z_1}}}{[(z_1-x_1)^2+|\mathbf z_2|^2\,]^{(d-2)/2}}\frac1{z_1}d\mathbf z_2\,dz_1
\end{align*}
We
 only need to show that the following three integrals are uniformly bounded for $x_1\geq 4$.
The first integral is
\begin{align*}
I_1&=\int_{x_1+\sqrt{x_1}}^\infty  \int_{|\mathbf z_2|<\sqrt{z_1}}
\frac{1}{|z_1-x_1|^{d-2}+|\mathbf z_2|^{d-2}} \frac{1}{z_1}\, d\mathbf z_2\, dz_1\\ 
& \leq \int_{x_1+\sqrt{x_1}}^\infty  \int_{|\mathbf z_2|<\sqrt{z_1}}
\frac{1}{|z_1-x_1|^{d-2}} \frac{1}{z_1}\, d\mathbf z_2\, dz_1\\
&= c \int_{x_1+\sqrt{x_1}}^\infty  
\frac{1}{|z_1-x_1|^{d-2}} \frac{1}{z_1} z_1^{\frac{1}{2}(d-1)}\, dz_1
 = c  \int_{\sqrt{x_1}}^\infty
\frac{1}{z_1^{d-2}} \ (z_1+x_1)^{\frac{d}{2}-\frac{3}{2}}\, dz_1\\
&\leq c \int_{\sqrt{x_1}}^{x_1}
\frac{1}{z_1^{d-2}} \ x_1^{\frac{d}{2}-\frac{3}{2}}\, dz_1 
+c \int_{x_1}^\infty 
\frac{1}{z_1^{d-2}} \ z_1^{\frac{d}{2}-\frac{3}{2}}\, dz_1 \leq c<\infty.
\end{align*}
The second integral we consider is
\begin{align*}
I_2 & = \int_2^{x_1-\sqrt{x_1}}  \int_{|\mathbf z_2|<\sqrt{z_1}}
\frac{1}{{|z_1-x_1|^{d-2}+|\mathbf z_2|^{d-2}}} \frac{1}{z_1}\, d\mathbf z_2\, dz_1\\
& \leq
\int_2^{x_1-\sqrt{x_1}}  \int_{|\mathbf z_2|<\sqrt{z_1}}
\frac{1}{{|z_1-x_1|^{d-2}}} \frac{1}{z_1}\, d\mathbf z_2\, dz_1
=c \int_2^{x_1-\sqrt{x_1}}  
\frac{1}{{(x_1-z_1)^{d-2}}} z_1^{\frac{d}{2} -\frac{3}{2}}\,  dz_1\\
&= c \int_{\sqrt{x_1}}^{x_1-2} \frac{1}{w^{d-2}} (x_1-w)^{\frac{d}{2}-\frac{3}{2}}\, dw
\leq c \int_{\sqrt{x_1}}^{x_1-2} \frac{1}{w^{d-2}} x_1^{\frac{d}{2}-\frac{3}{2}}\, dw
\le c<\infty.
\end{align*}
The remaining integral is
\begin{align*}
I_3
&=\int_{x_1-\sqrt{x_1}}^{x_1+\sqrt{x_1}}  \int_{|\mathbf z_2|<\sqrt{z_1}}
\frac{1}
{[(z_1-x_1)^2+|\mathbf z_2|^2\,]^{(d-2)/2}}\frac1{z_1}d\mathbf z_2\,dz_1\\
&\leq  2\int_{x_1-\sqrt{x_1}}^{x_1+\sqrt{x_1}}  \int_{|\mathbf z_2|<2\sqrt{x_1}}
\frac{1}
{[(z_1-x_1)^2+|\mathbf z_2|^2\,]^{(d-2)/2}}\frac1{z_1}d\mathbf z_2\,dz_1
\\
&\leq 2 \int_{B(0,\, 3 \sqrt{x_1})}\frac{1}{{|z|^{d-2}}} \frac{1}{x_1}\, d z
\leq c< \infty.
\end{align*} 
To prove the second statement of Proposition~\ref{thm:dist}, for $s>0$ we let 
$
{\it d}_s 
f(x)=sf(\sqrt{s}x)$. Note that the dilatation does not change the norms:
$$
\| \Delta^{-1}({\it d}_s f)  \|_{\infty}=\| \Delta^{-1} f \|_{\infty}\,, \qquad \|\Kz({\it d}_s f)\|_{\infty }=\|K(f)\|_{\infty}\,. 
$$
Furthermore,
${\rm supp} ({\it d}_s f) \subseteq B(0,r/\sqrt{s})$ if ${\rm supp} (f)\subseteq B(0,r)$, $r>0$.
Since $\|\Delta^{-1}V\|_{\infty}=C<\infty$ and $\|\Kz (V)\|_\infty=\infty$,
therefore $\|\Delta^{-1} (V\ind_{B_r})\|_{\infty}\leq C$  for every $r>0$ and $\|K(V\ind_{B_r})\|_{\infty }\to \infty$ as $r\to \infty$.
For $n\in \NN$ we define $$V_n={\it d}_{r_n^2}(V\ind_{B_{r_n}}) \,,$$ 
 where $r_n$ is chosen such that 
$
\|K(V\ind_{B_{r_n}})\|_{\infty }\geq 4^n
$.
Also, ${\rm supp}(V_n)\subseteq B(0,1)$.
We define $\tilde{V}=\sum_{n=1}^{\infty}V_n/2^n$. Then,
$$
\|K(\tilde{V})\|_{\infty }\geq \|K(V_n)\|_{\infty }/2^n\geq 2^n \to \infty\,,
$$
as $n \to \infty$,
and
$$
\|\Delta^{-1}\tilde{V}\|_{\infty}\leq \sum_{n=1}^{\infty} \|\Delta^{-1}V_n\|_{\infty}/2^n \leq C\,.
$$
\end{proof}
Similarly,  \eqref{est:sharp_uni} fails for $-\varepsilon \tilde V\ge 0$ with any $\varepsilon >0$, cf. 
Remark~\ref{rem:Vp}.

\section{Sufficient conditions for
the sharp Gaussian estimates}\label{sec:gen_app}
Recall from \cite[(2.5)]{MR549115} that for $p\in [1,\infty]$,
\begin{align*}
\| P_t f\|_{\infty} \leq C(d,p) \,t^{-d/(2p)} \|f \|_p\, ,\qquad t>0\,,
\end{align*}
where $P_t f(x)=\int_{\Rd} g(t,x,z)f(z)dz$, $f\in L^p(\Rd)$ and
$$
C(d,p)=\begin{cases} (4\pi)^{-d/2}, \quad & \mbox{ if \ \ } p=1\,,\\ (4\pi)^{-d/(2p)} (1-p^{-1})^{(1-p^{-1})d/2}, & \mbox{ if \ \ } p\in (1,\infty]\,. \end{cases}
$$
We will give an analogue for the {\it bridges} $T^{t,y}_s$. Here $t>0$, $y\in\Rd$, and
$$
T^{t,y}_s f (x) = \int_{\Rd} \frac{g(s,x,z)\,g(t-s,z,y)}{g(t,x,y)} f(z) \,dz\,,\qquad 0<s<t,\quad x\in\Rd\,.
$$
Clearly,
\begin{equation}\label{e:sT}
T^{t,y}_s f (x) =
T^{t,x}_{t-s} f (y),\qquad 0<s<t,\quad x,y\in\Rd\,.
\end{equation}
By  the Chapman-Kolmogorov equations (the semigroup property) for the kernel $g$, we have $T^{t,y}_s 1=1$.
We also note that $S(V)$ is related to the potential ($0$-resolvent) operator of $T$ as follows,
$$
S(V,t,x,y)=\int_0^t T^{t,y}_s\, |V|(x)\,ds\,.
$$
\begin{lem}
\label{lem:Lp}
For $p\in [1,\infty]$ and $f\in L^p(\Rd)$ we have
\begin{align*}
\| T^{t,y}_s f\|_{\infty} \leq C(d,p) \,\left[\frac{(t-s)s}{t}\right]^{-d/(2p)} \|f \|_p\,,\qquad 0<s<t,\,y\in\Rd\,.
\end{align*}
\end{lem}
\begin{proof}
We note that
$$\frac{g(s,x,z)\,g(t-s,z,y)}{(4\pi)^{-d/2}g(t,x,y)}=\left[\frac{(t-s)s}{t}\right]^{-\frac{d}{2}}
\exp\left[ -
\frac{|z-x|^2}{4s} -\frac{|y-z|^2}{4(t-s)} +\frac{|y-x|^2}{4t}
\right].
$$
As in \cite[(3.4)]{MR1457736}, we have
\begin{equation}\label{e:ti}
\frac{|z-x|^2}{4s}+\frac{|y-z|^2}{4(t-s)}\ge \frac{|y-x|^2}{4t}.
\end{equation}
Indeed, (\ref{e:ti}) follows from  the triangle  and Cauchy-Schwarz inequalities:
\begin{align*}
|y-x|&\le \sqrt{s}\frac{|z-x|}{\sqrt{s}}+\sqrt{t-s}\frac{|y-z|}{\sqrt{t-s}}
\le \sqrt{t}\left(
\frac{|z-x|^2}{s}+\frac{|y-z|^2}{t-s}
\right)^{1/2}.
\end{align*}
For $p=1$, the assertion of the lemma results from \eqref{e:ti}.
For $p\in (1,\infty)$, we let $p'=p/(p-1)$, apply H{\"o}lder's inequality, 
the identity
$g(s,x,z)^{p'}=g(s/p',x,z)(4\pi s)^{(1-p')d/2}(p')^{-d/2}$, and the semigroup property, to get
 \begin{equation*}\begin{split}
|T_s^{t,y}f(x)| &  \leq g(t,x,y)^{-1}\Big[\int_\Rd g(s,x,z)^{p'}g(t-s,z,y)^{p'}\, dz\Big]^{1\slash p'}\| f\|_p \\
&= g(t,x,y)^{-1} \Big[ (4\pi)^{(1-p')d}\,(p')^{-d}\, [s(t-s)]^{(1-p')d/2} \\
&\qquad \qquad \qquad \, \int_\Rd g(s/p',x,z)g((t-s)/p',z,y)\, dz \Big]^{1\slash p'}\| f\|_p\\
  &=C(d,p)\Big[\frac{s(t-s)}{t}\Big]^{-d\slash (2p)}\| f\|_p.
\end{split}\end{equation*}
For $p=\infty$, the assertion  follows from the identity $T^{t,y}_s 1=1$.
\end{proof}

Zhang \cite[Proposition~2.1]{MR1978999}
showed that \eqref{est:sharp_uni} and \eqref{est:sharp_time} hold for $V$ in specific $L^p$ spaces (see also \cite[Theorem~1.1 and Remark~1.1]{MR1978999}).
In Proposition~\ref{prop:Zhang_most} and Corollary~\ref{cor:esc} below we prove his result by a different method.
\begin{prop}\label{prop:Zhang_most}
Let $V\colon\Rd\to\RR$ and $p,q\in [1,\infty]$.
\begin{enumerate}
\item[\rm (a)]
If $V\in L^p(\Rd)$, $p> d/2$ and $c= C(d,p) \frac{[\Gamma(1-d/(2p))]^2}{\Gamma(2-d/p)} \|V \|_p$, then
\begin{align*}
\sup_{x,y\in\Rd} S(V,t,x,y)
\leq c\,t^{1-d/(2p)}\,, \qquad t>0\,.
\end{align*}
\item[\rm(b)] If $V\in L^p(\Rd)\cap L^q(\Rd)$ and $q<d/2<p$, then \eqref{e:sSbi} holds.
\end{enumerate}
\end{prop}
\begin{proof}
Part ${\rm(a)}$ follows from Lemma~\ref{lem:Lp}, so we proceed to ${\rm(b)}$.
For $t>2$,
\begin{align}\label{e:dT}
\int_0^t T^{t,y}_s |V|(x)\,ds = \int_0^{t/2} T^{t,y}_s |V|(x)\,ds+\int_0^{t/2} T^{t,x}_s |V|(y)\,ds\,.
\end{align}
By Lemma~\ref{lem:Lp}, the first term of the sum can be estimated as follows:
\begin{align}
\int_0^{t/2}  \!\!\!\!T^{t,y}_s |V|(x)\,ds
&\leq c\, \| V \|_p \int_0^1 \left[\frac{(t-s)s}{t}\right]^{-d/(2p)} \!\!\!\! \!\!\!\!\!\!\!\!ds+c\, \| V \|_q\int_1^{t/2} \left[\frac{(t-s)s}{t}\right]^{-d/(2q)}   \!\!\!\!\!\!\!\!\!\!\!\!ds
\nonumber
\\
&\leq  c'\, \| V \|_p \int_0^1 s^{-d/(2p)} ds+c'\, \| V \|_q\int_1^{\infty} s^{-d/(2q)}ds.
\label{eq:jpq}
\end{align}
By \eqref{e:sT}, the second term
has the same bound.
For $t\in (0,2]$ we use ${\rm(a)}$.
\end{proof}

Lemma~\ref{cor:gen_neg} and \ref{lem:gen_neg} provide the following conclusion:
\begin{cor}\label{cor:esc}
Under the assumptions of
{\rm Proposition~\ref{prop:Zhang_most}(a)}, $G$ 
satisfies the
sharp local Gaussian bounds \eqref{est:sharp_time}. 
If $V\le 0$ and the assumptions of
{\rm Proposition~\ref{prop:Zhang_most}(b)}
hold, then $G$ has the sharp global Gaussian bounds \eqref{est:sharp_uni}.
\end{cor}

Recall from Section~\ref{sec:prel}
that 
\eqref{est:sharp_uni} holds 
 for $d=3$ if $\|\Delta^{-1} V^- \|_{\infty}<\infty$ and $\|\Delta^{-1} V^+ \|_{\infty}<1$, and 
it holds  for $d\geq 4$ if
$\|\Delta^{-1} V^- \|_{\infty}+\|V^-\|_{d/2}<\infty$, $\|\Delta^{-1} V^+ \|_{\infty}<1$ and
$\|V^+\|_{d/2}$ is small enough. 
This yields another proof of the second statement of Corollary~\ref{cor:esc}, because of the following observation:
\begin{lem}\label{lem:Zhang_Lisk_Sem}
The
assumptions of {\rm Proposition~\ref{prop:Zhang_most}${\rm(b)}$}
necessitate that
$d\geq 3$, 
 $\|\Delta^{-1} |V| \|_{\infty}<\infty$ and $\|V\|_{d/2}<\infty$.
\end{lem}
\begin{proof}
Plainly, the assumptions of {\rm Proposition~\ref{prop:Zhang_most}${\rm(b)}$} imply
$d>2$ and $V\in L^{d/2}(\Rd)$. We now verify that \mbox{$\|\Delta^{-1} |V| \|_{\infty}<\infty$.} By H\"older's inequality,
\begin{align*}
\sup_{x\in\Rd}\int_{B(0,1)} \frac{|V(z+x)|}{|z|^{d-2}} dz\leq \||z|^{2-d}\ind_{B(0,1)}(z) \|_{p'}\, \|V \|_{p}<\infty\,,\\
\sup_{x\in\Rd}\int_{B^c(0,1)} \frac{|V(z+x)|}{|z|^{d-2}} dz\leq \||z|^{2-d}\ind_{B^c(0,1)}(z) \|_{q'}\, \|V \|_{q}<\infty\,,
\end{align*}
where $p',q'$ are the exponents conjugate to $p,q$, respectively.
\end{proof}

In what follows, we develop
sufficient conditions for \eqref{est:sharp_uni} and \eqref{est:sharp_time}.
Let $d_1,d_2\in \NN$ and $d=d_1+d_2$.
\begin{rem}\label{rem:obs}
\rm
The Gauss-Weierstrass kernel $g(t,x)$ in $\Rd$ can be represented as a tensor product:
\begin{align*}
g(t,x)= (4\pi t)^{-d_1/2} e^{-|x_1|^2/(4t)}\, (4\pi t)^{-d_2/2} e^{-|x_2|^2/(4t)}\,,
\end{align*}
where $x_1\in\RR^{d_1}$, $x_2\in\RR^{d_2}$ and $x=(x_1,x_2)$.
The kernels of the bridges factorize accordingly:
\begin{align*}
&\frac{g(s,x,z)\,g(t-s,z,y)}{g(t,x,y)}\\
&=\frac{(4\pi s)^{-d_1/2} e^{-|z_1-x_1|^2/(4s)}(4\pi (t-s))^{-d_1/2} e^{-|y_1-z_1|^2/(4(t-s))}}{(4\pi t)^{-d_1/2} e^{-|y_1-x_1|^2/(4t)}}\\
&\times\frac{(4\pi s)^{-d_2/2} e^{-|z_2-x_2|^2/(4s)} (4\pi (t-s))^{-d_2/2} e^{-|y_2-z_2|^2/(4(t-s))}}{(4\pi t)^{-d_2/2} e^{-|y_2-x_2|^2/(4t)}}.
\end{align*}
\end{rem}

\begin{cor}\label{rem:prod}
Let $V_1\colon \RR^{d_1}\to \RR$, $V_2\colon\RR^{d_2}\to \RR$,
and $V(x)=V_1(x_1)V_2(x_2)$, where $x=(x_1,x_2)\in \Rd$, $x_1\in \RR^{d_1}$ and $x_2\in \RR^{d_2}$.
Assume that $V_1\in L^{\infty}(\RR^{d_1})$ and $\sup_{t>0,\,x_2,y_2\in\RR^{d_2}} S(V_2,t,x_2,y_2)<\infty$.
Then \eqref{e:sSbi} holds.
\end{cor}
\begin{proof}
In estimating $S(V,t,x,y)$ we first use the factorization of the bridges and the boundedness of $V_1$, and then the Chapman-Kolmogorov equations and the boundedness of $S(V_2)$.
\end{proof}
\begin{lem}
\label{lem:Lp1_Lp2}
For $f (x_1,x_2)=f_1(x_1)f_2(x_2) \in L^{p_1}(\RR^{d_1})\times L^{p_2}(\RR^{d_2})$, $0<s<t$ and $y\in\Rd$,
we have
\begin{align*}
\| T^{t,y}_s f\|_{\infty} \leq C(d_1,p_1)\, C(d_2,p_2) \left[\frac{(t-s)s}{t}\right]^{-d_1/(2p_1)-d_2/(2p_2)} \|f_1 \|_{p_1} \|f_2 \|_{p_2}\,.
\end{align*}
\end{lem}
\begin{proof}
We proceed as in the proof of Lemma~\ref{lem:Lp}, using Remark~\ref{rem:obs}.
\end{proof}

\begin{proof}[Proof of Theorem~\ref{prop:most_extended}]
We follow
the proof of Proposition~\ref{prop:Zhang_most}, replacing Lemma~\ref{lem:Lp}
by Lemma~\ref{lem:Lp1_Lp2}.\end{proof}
We note in passing that Theorem~\ref{prop:most_extended} is an extension of
Proposition~\ref{prop:Zhang_most}.

\section{Examples}\label{sec:e}
Let $\ind_A$ denote the indicator function of $A$. In what follows, $G$ in \eqref{est:sharp_uni} is the Schr{\"o}dinger perturbation of $g$ by $V$.

\begin{example}\label{thm:Ld/2}
Let $d\geq 3$ and $1<p<\infty$.
For $x_1\in  \RR$, $x_2\in \RR^{d-1}$ we let
$V(x_1,x_2)=-
|x_1|^{-1/p}\ind_{|x_1|<1}\ind_{|x_2|<1}$.
Then \eqref{est:sharp_uni} holds but $V \notin L^{p}_{loc}(\Rd)$.
\end{example}
\noindent
Indeed,
$V(x_1,x_2)=V_1(x_1) V_2(x_2)$, where
\begin{align*}
V_1(x_1)&=-|x_1|^{-1/p} \ind_{|x_1|<1},\qquad x_1\in  \RR,\\
V_2(x_2)&=\ind_{|x_2|<1}, \qquad x_2\in \RR^{d-1}.
\end{align*}
Let
$$
1\le q<p_1<r<p,
$$
and
$$
p_2=\frac{d-1}{2}\frac{p_1}{p_1-1/2}.
$$
Since $d\ge 3$, $p_2>1$.
In the notation of Theorem~\ref{prop:most_extended} we have
\mbox{$d_1=1$}, \mbox{$d_2=d-1$}, and indeed
\mbox{$d_1/(2p_1)+d_2/(2p_2)=1$}.
Since $V_1\in L^r(\RR)\cap L^{q}(\RR)$ and $V_2\in L^{p_2}(\RR^{d-1})$,
the assumptions of Theorem~\ref{prop:most_extended}${\rm(b)}$ are satisfied,
and \eqref{est:sharp_uni} follows
by Corollary~\ref{cor:escb}.
Clearly, $V\notin L^{p}_{loc}(\Rd)$.

\begin{example}\label{ex:drugi}
For $d\geq 3$, $n=2,3,\ldots$, let $V_n(x)=|x_1|^{-1+1/n}\ind_{|x_1|<1} \ind_{|x_2|<1}$, where $x=(x_1,x_2)$, $x_1\in\RR$, $x_2\in \RR^{d-1}$. Let $a_n=\sup_{t>0,\, x,y\in\Rd} S(V_n,t,x,y)$,
$$
V(x)=-\sum_{n=2}^{\infty} \frac1{n^2}\frac{V_n(x)}{a_n}\,, \quad x\in \Rd.
$$
Then
\eqref{est:sharp_uni} holds but $V\notin \bigcup_{p> 1} L^p_{loc}(\Rd)$.
\end{example}
\noindent
Indeed,
$0<a_n<\infty$ by Example~\ref{thm:Ld/2}, and so
$$
\sup_{t>0,\,x,y\in\Rd} S(V,t,x,y)\leq
\sum_{n=2}^{\infty}\frac{1}{n^2} <\infty\,.
$$
This yields the global sharp Gaussian bounds.
Fix  $p>1$. Since the function $x_1\to |x_1|^{-1+1/n}$ is not in $L^p(-1,1)$ for large $n$, we get that $V\notin L^p(B(0,1))$.

\begin{example}\label{ex:czwarty}
Let $d\ge 3$ and $V(x_1,x_2)=\frac{-1}{(|x_2|+1)^3}$ for $x_1\in\RR^{d-3}$, $x_2\in \RR^{3}$.
Then \eqref{est:sharp_uni} holds but $V\notin L^1(\RR^d)$.
\end{example}
\noindent
Indeed, $V\notin L^1(\RR^d)$.
We let $V_2(x_2)=\frac{-1}{(|x_2|+1)^3}$, $x_2\in \RR^3$. By the symmetric rearrangement inequality \cite[Chapter~3]{MR1817225},
in dimension $d=3$ we have
\begin{align*}
0\le \Delta^{-1}V_2
\le C_{3}\int_{\RR^3} \frac{1}{|z|(|z|+1)^{3}} \,dz <\infty\,.
\end{align*}
By \eqref{eq:cSK0} and \eqref{eq:d_3},
$$
\sup_{t>0,\,x_2,y_2\in\RR^3} S(V_2,t,x_2,y_2)<\infty.
$$
By Corollary~\ref{rem:prod} and Lemma~\ref{cor:gen_neg} we see that \eqref{est:sharp_uni} holds for $V$.

\begin{proof}[Proof of Proposition~\ref{cor:ce}]
Add the functions from
Example~\ref{ex:drugi} and 
~\ref{ex:czwarty}.
\end{proof}
We can have nonnegative examples, too. Namely, let $V\le 0$ be as in Proposition~\ref{cor:ce}.
Then $M=\sup_{t>0,x,y\in \Rd}S(V,t,x,y)<\infty$.
We
let $\tilde{V}= |V|/(M+1)$.
Then $\tilde{V}\ge 0$, $\tilde{V}\notin L^1(\Rd)\cup \bigcup_{p>1}L^p_{loc}(\Rd)$
and
$$
\sup_{t>0,\,x,y\in\Rd} S(\tilde{V},t,x,y)=M/(M+1)<1\,.
$$
Therefore \eqref{gen_est} holds for $\tilde{V}$ with $h=\infty$ and $\eta=M/(M+1)$, which  yields \eqref{est:sharp_uni}.

Let $d_1,d_2\in \NN$, $d=d_1+d_2$, $V_1\colon \RR^{d_1}\to \RR$, $V_2\colon\RR^{d_2}\to \RR$,
and $V(x_1,x_2)=V_1(x_1)+V_2(x_2)$, where $x_1\in \RR^{d_1}$ and $x_2\in \RR^{d_2}$.
Let $G_1(t,x_1,y_1)$, $G_2(t,x_2,y_2)$ be the 
Schr\"odinger perturbations of the Gauss-Weierstrass kernels on $\RR^{d_1}$ and $\RR^{d_2}$ by $V_1$ and $V_2$, respectively.
Then
$G(t,(x_1,x_2),(y_1,y_2)):= G_1(t,x_1,y_1) \allowbreak G_2(t,x_2,y_2)$ is the Schr\"odinger perturbation of the Gauss-Weierstrass kernel  on $\RR^d$ by $V$.
Clearly, if 
the sharp Gaussian estimates hold for $G_1$ and $G_2$, then they hold for $G$.
Our next example 
shows that the situation
is quite different for tensor products.

\begin{example}\label{ex:nfS2}
Let $V(x_1,x_2)= V_1(x_1)V_2(x_2)$, where $x_1,x_2\in \RR^3$, 
$$
V_1(x)=V_2(x)=-\frac{1-\varepsilon}{2}\ |x|^{-1-\varepsilon}\ \ind_{|x|<1}\,,
$$
and $\varepsilon\in [0,1)$.
Then the heat kernels in $\RR^3$ of
$\Delta+V_1$ and $
\Delta+V_2$ satisfy \eqref{est:sharp_uni} and \eqref{est:sharp_time}, but 
that of $
\Delta+V$
in $\RR^6$ 
satisfies neither \eqref{est:sharp_uni} nor \eqref{est:sharp_time}.
\end{example}
\noindent
Indeed, by the symmetric rearrangement inequality \cite[Chapter~3]{MR1817225}, 
\begin{align*}
0\le-\Delta^{-1}  V_1  (x)\leq
-\Delta^{-1} V_1  (0)=
\frac{1-\varepsilon}{8\pi} \int_{\{z\in\RR^3:|z|<1\}}\frac{1}{|z|} |z|^{-1-\varepsilon}\,dz= 1/2, 
\end{align*}
for all $x\in \RR^3$.
Thus, $\|\Delta^{-1}V_1\|_{\infty}=\| \Delta^{-1} V_2\|_{\infty}<\infty$. 
Using the comment following \eqref{eq:d_3},
 we get
\eqref{est:sharp_uni}
for the heat kernels  in $\RR^3$ of 
$\Delta+V_1$ and $
\Delta+V_2$.
However, the heat kernel in $\RR^6$ of
$
\Delta+V$
fails even \eqref{est:sharp_time}.
Indeed,
if we let $T\leq 1$, $a\in\RR^6$, $|a|=1$, and
$c=\int_0^1 p(s,0,a)ds$, then by \cite[Lemma~3.5]{MR1329992},
\begin{align*}
&\int_0^T \int_{\RR^6}  g(s,0,x)|V(x)|\,dxds
\geq
\int_{\{x \in\RR^6:|x|^2\leq T\}} \int_0^T g(s,0,x)ds\,|V(x)|\,dx\\
&\geq
c
\int_{\{x \in\RR^6:|x|^2\leq T\}}  \frac1{|x|^4} |V(x)|\,dx
\\
&\geq c \int_{\{x_1\in\RR^3:|x_1|^2<T/2\}}  |V_1(x_1)|
\int_{\{x_2\in\RR^3:|x_2|^2<T/2\}}
\frac{|V_2(x_2)|}{(|x_1|^2+|x_2|^2)^2} \,dx_2 dx_1
\\
&\geq  \frac{c (1-\varepsilon)}{2} \int_{\{x_1\in\RR^3:|x_1|^2<T/2\}} |V_1(x_1)|
\int_{\{x_2\in\RR^3:|x_2|^2<T/2\}}
\frac{|x_2|^{-1}}{(|x_1|^2+|x_2|^2)^2} \,dx_2 dx_1
\\
&= \frac{c (1-\varepsilon)}{2} \int_{\{x_1\in\RR^3:|x_1|^2<T/2\}}  |V_1(x_1)|
 \frac{\pi T}{|x_1|^2 (T/2+|x_1|^2)}\,dx_1 \\
&=  \pi^2 c\,T (1-\varepsilon)^2 \int_0^{\sqrt{T/2}} \frac{r^{-1-\varepsilon}}{T/2+r^2}\,dr=\infty\,.
\end{align*}
Therefore by  Lemma~\ref{l:b}, 
 \eqref{e:sSbt} fails, and so does \eqref{est:sharp_time},
according to Remark~\ref{rem:Vp}.
Thus, the sharp Gaussian estimates may hold for the Schr\"odinger perturbations of the Gauss-Weierstrass kernels by $V_1$ and $V_2$ but fail for the Schr\"odinger perturbation of the Gauss-Weierstrass kernel by $V(x_1,x_2)=V_1(x_1)V_2(x_2)$. 

In passing we note that the functions $-V_1$, $-V_2$ and $-V$
give a similar counterexample 
with nonnegative factors, because $1/2<1$, cf. \eqref{eq:mn1}.
Let us also remark that the sharp global Gaussian estimates may 
hold for $V(x_1,x_2)=V_1(x_1)V_2(x_2)$ but fail for $V_1$ or $V_2$. 
Indeed, it suffices to consider $V_1(x_1)=-\ind_{|x_1|<1}$ on $\RR^3$ and $V_2\equiv 1$ on $\RR$, and to apply Theorem~\ref{prop:most_extended}. We see that it is  indeed the combined effect of the factors $V_1$ and $V_2$ that matters--as captured in Section~\ref{sec:gen_app}.

\section*{Acknowledgement}
Krzysztof Bogdan was supported by the Polish National Science Center
(Narodowe Centrum Nauki, NCN) grant 2014/14/M/ST1/00600.
Jacek Dziuba{\'n}ski was supported by the NCN grant DEC-2012/05/B/ST1/00672.
Karol Szczypkowski was partially supported by IP2012 018472
and by the German Science Foundation (SFB 701). 
We thank the referee for insightful comments and suggestions, which largely shaped the paper. In particular 
the paper merges the results of  two preprints \cite{2015arXiv151107167B} and \cite{2016arXiv160603745B},
and the proof of Theorem~\ref{thm:J_0} is much shorter than the elementary
arguments given in \cite{2016arXiv160603745B}.

\def\polhk#1{\setbox0=\hbox{#1}{\ooalign{\hidewidth
  \lower1.5ex\hbox{`}\hidewidth\crcr\unhbox0}}}
  \def\polhk#1{\setbox0=\hbox{#1}{\ooalign{\hidewidth
  \lower1.5ex\hbox{`}\hidewidth\crcr\unhbox0}}}

\end{document}